\documentclass{article}
\usepackage{amssymb}
\usepackage{amsmath}
\usepackage{graphicx}

\newcommand {\stirlingf}[2]{\genfrac[]{0pt}{}{#1}{#2}}
\newcommand {\stirlings}[2]{\genfrac\{\}{0pt}{}{#1}{#2}}
\newcommand {\lah}[2]{\genfrac\lfloor \rfloor{0pt}{}{#1}{#2}}

\newtheorem{theorem}{Theorem}

\newenvironment{proof}[1][Proof]{\noindent\textbf{#1.} }{\ \rule{0.5em}{0.5em}}
\include {mak}
\begin{document}

\title{Combinatorial approach of certain generalized Stirling numbers}
\author{Hac\`{e}ne Belbachir, Amine Belkhir and Imad Eddine Bousbaa  \\
USTHB, Faculty of Mathematics\\
RECITS Laboratory, DG-RSDT\\
BP 32, El Alia, 16111, Bab Ezzouar, Algiers, Algeria\\
hbelbachir@usthb.dz \& ibousbaa@usthb.dz}
\maketitle
\date{}
\begin{abstract}
A combinatorial methods are used to investigate some properties of certain
generalized Stirling numbers, including explicit formula and recurrence
relations. Furthermore, an expression of these numbers with symmetric
function is deduced.
\end{abstract}



\section{Introduction}

The Stirling numbers of the first kind $\stirlingf{n}{k}$, the second kind $%
\stirlings{n}{k}$ and the Third kind $\lah{n}{k}$ which are know as Lah numbers. These numbers are coefficients in the
expression of the raising-falling factorial of $x$, see \cite[pp 204]%
{MR1999993},

\begin{eqnarray}
\left( x\right) \overline{^{n}} &=&\sum_{k=0}^{n}\stirlingf{n}{k}x^{k}, \\
x^{n} &=&\sum_{k=0}^{n}\stirlings{n}{k}(x)^{\underline{k}},
\end{eqnarray}

\begin{equation*}
\left( x\right) \overline{^{n}} =\sum_{k=0}^{n}\lah{n}{k}(x)^{\underline{k}},
\end{equation*}
where $\left( x\right) \overline{^{n}}=x(x+1)\cdots (x+n-1)$ and $(x)^{\underline{n}}=x(x-1)\cdots (x-n+1)$, with $\left( x\right) \overline{^{0}}=(x)^{\underline{0}}=1$.

These three sequences satisfy respectively the following recurrence relations%
\begin{eqnarray}
\stirlingf{n}{k} &=&\stirlingf{n-1}{k-1}+\left( n-1\right) \stirlingf{n-1}{k},  \label{recS1} \\
\stirlings{n}{k} &=&\stirlings{n-1}{k-1}+k\stirlings{n-1}{k},  \label{recS22} \\
\lah{n}{k} &=&\lah{n-1}{k-1}+\left(n+k-1\right) \lah{n-1}{k}.  \label{recLah}
\end{eqnarray}%
with $\stirlingf{n}{0}=\stirlings{n}{0}=\lah{n}{0}=\delta _{n,0}$, where $\delta $ is the Kronecker delta, and for $n\neq 0$ $\stirlingf{n}{k}=\stirlings{n}{k}=\lah{n}{k}=0$ when $k\notin \{0,\ldots,n\}$.

The Stirling numbers of the both kinds and the Lah numbers have a
combinatorial interpretation, see for instance \cite{BelBou14,BelBou12,BelBel13}%
, as follows
\begin{itemize}
\item $\stirlingf{n}{k}$ count the
number of permutations of the set $\{1,\ldots ,n\}$ with $k$ cycles,

\item $\stirlings{n}{k}$ count
the number of partitions of the set $\{1,\ldots ,n\}$ into $k$ subsets,

\item $\lah{n}{k}$ count the number of
partitions of the set $\{1,\ldots ,n\}$ into $k$ ordered lists.
\end{itemize}

Many generalizations of the Stirling numbers were developed using:
combinatorial approach, see Howard \cite{MR600368}, Broder \cite{MR743795};
the falling factorial, see Belbachir et \textit{al} \cite{MR25}, Benoumhani
\cite{MR1415279}; or generating functions, see Carlitz \cite{MR570168,
MR531621}.

As a unified approach to the generalized Stirling numbers, Hsu and Shiue
\cite{MR1618435}, defined a Stirling-type pair $\left\{ S^{1}\left(
n,k\right) ,S^{2}\left( n,k\right) \right\} \equiv \left\{ S\left(
n,k;\alpha ,\beta ,r\right) ,S\left( n,k;\beta ,\alpha ,-r\right) \right\} $
by the inverse relations%
\begin{equation}
\left( x|\alpha \right)^{\underline{n}}=\sum_{k=0}^{n}S^{1}\left( n,k\right) \left(
x-r|\beta \right)^{\underline{k}},
\end{equation}%
\begin{equation}
(x|\beta )^{\underline{n}}=\sum_{k=0}^{n}S^{2}\left( n,k\right) \left( x+r|\alpha
\right)^{\underline{k}},
\end{equation}
where $\alpha ,\beta $ and $r$ are arbitrary parameters with $\left( \alpha
,\beta ,r\right) \neq (0,0,0)$ and $\left( x|\theta \right) ^{\underline{n}}$
is the generalized $n$-$th$ falling factorial of $x$ with increment $\theta $
defined by
\begin{eqnarray*}
\left( x|\theta \right) ^{\underline{n}} &=&x\left( x-\theta \right) \left(
x-2\theta \right) \cdots \left( x-\theta \left( n-1\right) \right) ,\text{ \
\ }n\geq 1, \\
\left( x|\theta \right) ^{\underline{0}} &=&1.
\end{eqnarray*}
The numbers $S\left( n,k;\alpha ,\beta ,r\right) $ satisfy the following
triangular recurrence relation
\begin{equation*}
S\left( n,k;\alpha ,\beta ,r\right) =S\left( n-1,k-1;\alpha ,\beta ,r\right)
+\left( \beta k-\alpha n+r\right) S\left( n-1,k;\alpha ,\beta ,r\right) ,
\end{equation*}
with $S\left( 0,0;\alpha ,\beta ,r\right) =S\left( n,n;\alpha ,\beta
,r\right) =1$ and $S\left( 1,0;\alpha ,\beta ,r\right) =r.$

Tsylova \cite{Tsy4285} gave a partial statistical interpretation of the numbers $%
A_{\beta ,\alpha }\left( k,n\right) $ which coincide with the special case $%
S\left( n,k;\alpha ,\beta ,0\right) $, also Belbachir and Bousbaa \cite{MR25}
define the translated Whitney numbers of the three kinds using a
combinatorial approach: the translated Whitney numbers of the first kind,
denoted $\stirlingf{n}{k}^{(\alpha )}$, count the number of
permutations of $n$ elements with $k$ cycles such that the elements of each
cycle can be colored in $\alpha $ ways except the smallest one; the
translated Whitney numbers of second kind, denoted $\stirlings{n}{k}^{(\alpha )}$,
count the partitions of the set $\{1,2,...,n\}$ into $k$
subsets such the elements of each subset can mute in $\alpha $ ways except
the smallest one; and the translated Whitney-Lah numbers, denoted $%
\lah{n}{k}^{(\alpha )}$, count the number of ways to
distribute the set $\{1,2,...,n\}$ into $k$ ordered lists such that the
elements of each list can mute with $\alpha $ ways, except the dominant one.
These three kinds of numbers correspond, respectively, to $\stirlingf{n}{k}%
^{(\alpha )}=S(n,k;-\alpha ,0,0)$, $\stirlings{n}{k}^{(\alpha
)}=S(n,k;0,\alpha ,0)$ and $\lah{n}{k}^{(\alpha
)}=S(n,k;-\alpha ,\alpha ,0)$. Motivated by the previews works, our aim is
to find combinatorial meaning for the numbers $S\left( n,k,\alpha
,\beta ,0\right) $ which are denoted by $\lah{n}{k}^{\alpha ,\beta }$. We start by giving the combinatorial
interpretation in the first section, also we prove combinatorially the
explicit formula using the inclusion and exclusion principle. In sections 2
and 3, we derive some recurrence relations and an expression using the
symmetric functions. We give, in the last section, a convolution identity.

\section{Combinatorial Interpretation and explicit formula}
Let $\Omega _{n,k}$ to be the set of all possible ways to distribute $n$
elements, denoted $1,2,\ldots ,n$,\ into $k$ ordered \textit{no empty}
lists, one element at a time, such that:

\begin{description}
  \item[({\romannumeral 1})]  we assign a weight of $\beta $ to the head list,
  \item[({\romannumeral 2})]  the remaining elements in the list have weight $\alpha$,
  \item[({\romannumeral 3})]  the first element putted in the list have a weight $1$.
\end{description}

Given a distribution $\varepsilon \in \Omega _{n,k},$ we define the weight
of $\varepsilon $, denoted by $w\left( \varepsilon \right) $, to be the
product of the weights of its elements. Since the total weight of $\Omega
_{n,k}$ is given by the sum of weights of all distributions.

To clarify the interpretation given above, we illustrate the case when $n=3$
and $k=1.$ First, the weight of the first element is $w\left( \left\{ \left(
1\right) \right\} \right) =1.$ Next, there are two ways to add the second
element either after the first element $\varepsilon _{1}=\left\{ \left(
1,2\right) \right\} $ with $w\left( \varepsilon _{1}\right) =\alpha $ or
before the first element $\varepsilon _{2}=\left\{ \left( 2,1\right)
\right\} $ with $w\left( \varepsilon _{2}\right) =\beta .$ Finally, to add
the third element
\begin{gather*}
\varepsilon _{11}=\left\{ \left( 1,2,3\right) \right\} \dashrightarrow
w\left( \varepsilon _{11}\right) =\alpha ^{2},\text{ \ }\varepsilon
_{12}=\left\{ \left( 1,3,2\right) \right\} \dashrightarrow w\left(
\varepsilon _{12}\right) =\alpha ^{2},\text{ } \\
\text{\ \ }\varepsilon _{13}=\left\{ \left( 3,1,2\right) \right\}
\dashrightarrow w\left( \varepsilon _{13}\right) =\alpha \beta ,\text{ \ }%
\varepsilon _{21}=\left\{ \left( 2,1,3\right) \right\} \dashrightarrow
w\left( \varepsilon _{21}\right) =\alpha \beta ,\text{ \ \ \ } \\
\varepsilon _{22}=\left\{ \left( 2,3,1\right) \right\} \dashrightarrow
w\left( \varepsilon _{22}\right) =\alpha \beta ,\text{ \ }\varepsilon
_{23}=\left\{ \left( 3,2,1\right) \right\} \dashrightarrow w\left(
\varepsilon _{23}\right) =\beta ^{2}.\text{ \ \ }
\end{gather*}%
Thus, the total weight of $\Omega _{3,1}$ is $2\alpha ^{2}+3\alpha \beta
+\beta ^{2}=\left( \alpha +\beta \right) \left( 2\alpha +\beta \right) .$

This approach suggests us the following result.
\begin{theorem}
For any non-negative integers $n$ and $k$, we have
\begin{equation}
\lah{n}{k}^{\alpha ,\beta }=\sum_{\varepsilon \in \Omega
_{n,k}}w\left( \varepsilon \right) .
\end{equation}
\end{theorem}

It is clear that%
\begin{equation}
\lah{n}{1}^{\alpha ,\beta
}=\prod\limits_{j=1}^{n-1}\left( j\alpha +\beta \right) . \label{h}
\end{equation}
The following result gives an explicit formulation of $\lah{n}{k}^{\alpha,\beta}$. The proof is based on inclusion-exclusion principle.
Without loose the generality, we can suppose $\alpha,\beta \in $ (or $\in \Re$).

\begin{theorem}
For any non-negative integers $n,k$, we have
\begin{equation}
\lah{n}{k}^{\alpha ,\beta }=\frac{1}{\beta ^{k}k!}%
\sum_{j=0}^{k}\left( -1\right) ^{j}\dbinom{k}{j}\left( \beta (k-j)|\alpha
\right) \overline{^{n}}.
\end{equation}
where $\left( \beta \left( k-j\right)|\alpha \right) ^{\overline{n}}= \beta \left(
k-j\right) \left( \beta \left( k-j\right) +\alpha \right) \cdots \left(
\beta \left( k-j\right) +(n-1)\alpha \right).$
\end{theorem}

\begin{proof}
Let $\phi$ denote the set of all possible ways to distribute $n$
elements, denoted $1,2,\ldots ,n$,\ into $k$ ordered lists (labeled and not
necessary non-empty), one element at a time, such that:

- the first element putted in the list have a weight $\beta $,

- we assign a weight of $\beta $\ to the element inserted as head list,

- the remaining elements in the list have a weight $\alpha $.

The total weight of set $\phi $ is the sum of all weights of all the
distributions. There are three differences with the initial interpretation:
the first one concerns the no empty lists and the second one concerns the
weight $\beta $ assigned to the first element putted in each list and the
third one is the order between the lists. the two last differences will be
considered at the end of the proof.

Now, let $\Delta $ be the subset of $\phi $ which have non-empty lists.
We want to count the total weight of subset $\Delta $.

For $j$ $\left( 1\leq j\leq k\right), $let $A_{j}$ be the subset of $k$
labeled lists of $\phi $ such that the $j$-$th$ list\ is empty. Then
\begin{equation*}
\Delta =\overline{A_{1}}\cap \overline{A_{2}}\cap \cdots \cap \overline{A_{k}},
\end{equation*}
where
\begin{equation*}
\overline{A_{j}}=\phi \backslash A_{j}\text{ \ and \ }\left\vert
\overline{A_{1}}\cap \overline{A_{2}}\cap \cdots \cap \overline{A_{k}}%
\right\vert =\left\vert \phi \right\vert -\left\vert A_{1}\cup A_{2}\cup
\cdots \cup A_{k}\right\vert .
\end{equation*}
Applying inclusion-exclusion principle, we get%
\begin{equation}
\left\vert \Delta \right\vert =\left\vert \phi \right\vert
-\sum_{j=1}^{k}\left( -1\right) ^{j}\sum_{1\leq i_{1}<i_{2}<\cdots
<i_{j}\leq k}\left\vert A_{i_{1}}\cap A_{i_{2}}\cap \cdots \cap
A_{i_{j}}\right\vert .
\end{equation}%
We compute the general term $\sum_{1\leq i_{1}<i_{2}<\cdots <i_{j}\leq
k}\left\vert A_{i_{1}}\cap A_{i_{2}}\cap \cdots \cap A_{i_{j}}\right\vert $,
for a fixed $j$, there are $\tbinom{k}{j}$ ways to select $j$ empty lists
from $k$. And to distribute $n$ elements in the remaining $k-j$ lists, so
the first element have $k-j$ choices with the weight $\beta $ which gives a total weight of $\beta \left( k-j\right)$,
the second one have a weight of $\beta \left( k-j\right)+\alpha $
 coming from the $\left( k-j\right) $ choices as head
list with weight $\beta $ or after the inserted element with weight $\alpha $
and so on until the last element which has a weight of $\beta \left( k-j\right)
+(n-1)\alpha $. So, the total weight of this distribution is $\beta \left(
k-j\right) \left( \beta \left( k-j\right) +\alpha \right) \cdots \left(
\beta \left( k-j\right) +(n-1)\alpha \right) =\left( \beta \left( k-j\right)
|\alpha \right) ^{\overline{n}}.$ Thus%

\begin{equation*}
\sum_{1\leq i_{1}<i_{2}<\cdots <i_{j}\leq k}\left\vert A_{i_{1}}\cap
A_{i_{2}}\cap \cdots \cap A_{i_{j}}\right\vert =\dbinom{k}{j}\left( \beta
\left( k-j\right) |\alpha \right) ^{\overline{n}},
\end{equation*}%
and we get%
\begin{equation*}
\left\vert \Delta \right\vert =\sum_{j=0}^{k}\left( -1\right) ^{j}\dbinom{k}{%
j}\left( \beta (k-j)|\alpha \right) ^{\overline{n}}.
\end{equation*}

We divide by $k!$ to avoid the repeated permutations and by $\beta ^{k}$ to
give the $k$ first elements inserted in the lists the weight $1$.
\end{proof}

\section{Recurrence relations}
In this section, we give combinatorial proofs of the three types of
recurrence relations : the triangular, the horizontal and the vertical
recurrence relation.

\begin{theorem}
The numbers $\lah{n}{k}^{\alpha ,\beta }$ satisfy the
following triangular recurrence relation%
\begin{equation}
\lah{n}{k}^{\alpha ,\beta }=\lah{n-1}{k-1}^{\alpha ,\beta }+\left( \alpha \left( n-1\right) +\beta
k\right) \lah{n-1}{k}^{\alpha ,\beta },
\end{equation}%
where $\lah{n}{n}^{\alpha ,\beta }=1$ and $\lah{n}{k}^{\alpha ,\beta }=0$ for $k\in\{0,\ldots,n \}.$
\end{theorem}

\begin{proof}
We count the total distribution weight of the set $\left\{ 1,2,\ldots
,n\right\} $ into $k$ ordered, non-empty, lists according to the situation
of the last element $\left\{ n\right\} $.

$\bullet $ \ If $\left \{ n\right \} $ is a singleton (with weight $1$), the
remaining $n-1$ elements have to be distribute into $k-1$ ordered lists with
weight $\lah{n-1}{k-1}^{\alpha ,\beta }$.

$\bullet $ \ If $\left\{ n\right\} $ is not a singleton, the element $n$ is
in one of the $k$ lists with some other elements. Total weight of
distributing set $\left\{ 1,2,\ldots ,n-1\right\} $ into $k$ ordered no
empty lists is $\lah{n-1}{k}^{\alpha ,\beta }$, and
there are $n-1$ choices to insert the element $n$ after any of the elements $%
1,2,\ldots ,n-1$ with weight $\alpha $ and $k$ choices to insert the element
$n$ as head list with weight $\beta $. Hence, the weight is $\left( \alpha
\left( n-1\right) +\beta k\right) \lah{n-1}{k}^{\alpha
,\beta }.$
\end{proof}

In particular, for $\left( -\alpha ,1\right) $ and $\left( -1,\beta \right) $%
, the Theorem $3$ will reduce to triangular recurrence relation for
degenerate Stirling numbers \cite{MR531621}. Furthermore, for $\left( \alpha
,0\right) ,\left( 0,\beta \right) $ and $\left( \alpha ,\alpha \right) ,$ we
obtain the triangular recurrence relation of the translated Whitney numbers
of the both kinds and of the Whitney-Lah numbers \cite{MR25} respectively.

Next, we give a horizontal recurrence relation using combinatorial proof.

\begin{theorem}
For non negative integers $n,k$, we have%
\begin{equation}
\lah{n}{k}^{\alpha ,\beta }=\sum_{j=0}^{n-k}\left(
-1\right) ^{j}\left( \left( k+1\right) \beta +n\alpha |\alpha \right) ^{%
\overline{j}}\lah{n+1}{k+j+1}^{\alpha ,\beta }.
\end{equation}
\end{theorem}

\begin{proof}
We proceed by construction, the total weight of distributing a set $\left\{
1,2,\ldots ,n\right\} $ into $k$ ordered lists can be obtained from a total
weight of distributing the set $\left\{ 1,2,\ldots ,n+1\right\} $ into $k+1$
ordered lists \ excluding the weight of distributions which does not contain
element $n+1$ as singleton. Then, we obtain $\lah{n}{k}^{\alpha ,\beta }=$ $\lah{n+1}{k+1}^{\alpha ,\beta
}-\left( \left( k+1\right) \beta -n\alpha \right) \lah{n}{k+1}^{\alpha ,\beta }.$

Now, $\lah{n}{k+1}^{\alpha,\beta }$corresponds to the total weight of distribution a set $\left\{
1,2,\ldots ,n\right\} $ into $k+1$ ordered lists, it can be obtained from a
total weight of distribution set $\left\{ 1,2,\ldots ,n+1\right\} $ into $%
k+2 $ ordered lists excluding the weight of distributions which does not
contain element $n+1$ as singleton, which gives:
\begin{eqnarray*}
\lah{n}{k}^{\alpha ,\beta } &=&\lah{n+1}{k+1}^{\alpha ,\beta }-\left( \left( k+1\right) \beta -n\alpha \right)
\left( \lah{n+1}{k+2}^{\alpha ,\beta }-\left( \left(
k+2\right) \beta -n\alpha \right) \lah{n}{k+2}^{\alpha,\beta }\right) \\
&=&\lah{n+1}{k+1}^{\alpha ,\beta }-\left( \left(
k+1\right) \beta -n\alpha \right) \lah{n+1}{k+2}^{\alpha
,\beta }+\left( \left( k+1\right) \beta -n\alpha \right) \left( \left(
k+2\right) \beta -n\alpha \right) \lah{n}{k+2}^{\alpha
,\beta } \\
&&\vdots \text{ \ } \\
&=&\sum_{j=0}^{n-k}\left( -1\right) ^{j}\left( \left( k+1\right) \beta
+n\alpha |\alpha \right) ^{\overline{j}}\lah{n+1}{k+j+1}^{\alpha ,\beta }.
\end{eqnarray*}
\end{proof}

The last theorem in this Section is a vertical recurrence relation with combinatorial proof.
\begin{theorem}
Let $n$ and $k$ be non negative integers,  we have%
\begin{equation}
\lah{n+1}{k+1}^{\alpha ,\beta }=\sum_{i=k}^{n}\left(
\alpha +\beta |\alpha \right) ^{\overline{n-i}}\dbinom{n}{i}\lah{i}{k}^{\alpha ,\beta }.
\end{equation}
\end{theorem}

\begin{proof}
Let us consider the $i$ $(k\leqslant i\leqslant n)$ elements not in the same list of the element $n+1$. We have $\binom{n}{i}$ ways to choose the $i$ elements and the total weight to constitute the $k$ lists is $\lah{i}{k}^{\alpha ,\beta }$, the remaining $n-i+1$ elements belong to the same list and we have $%
\lah{n+1-i}{1}^{\alpha ,\beta }=\prod_{j=1}^{n-i}\left( j\alpha +\beta \right) $ (see \ref{h}). We conclude by summing.
\end{proof}

For $\left( \alpha ,\beta \right) =\left( 1,0\right) ,$ $\left( 0,1\right) $
and $\left( 1,1\right) $ we get the identities \cite[eq. 30]{MR743795}, \cite%
[eq. 35]{MR743795} and \cite[eq. 11]{BelBou14} respectively.

\section{Relation with symmetric functions}

The generalized Stirling numbers $\lah{n+k}{n}^{\alpha
,\beta }$, for fixed $n$, are the elementary symmetric functions of the
numbers $1,\ldots ,n$.

\begin{theorem}
For non negative integers $n,k,\alpha ,\beta ,$ we have%
\begin{eqnarray*}
\lah{n+k}{n}^{\alpha ,\beta } &=&\sum_{1\leqslant
i_{1}\leqslant \cdots \leqslant i_{k}\leqslant
n}\prod\limits_{j=1}^{k}\left( \left( \alpha +\beta \right) i_{j}+\alpha
\left( j-1\right) \right) , \\
&=&\sum_{1\leqslant i_{1}\leqslant \cdots \leqslant i_{k}\leqslant n}\left(
\alpha +\beta \right) i_{1}\left( \left( \alpha +\beta \right) i_{2}+\alpha
\right) \cdots \left( \left( \alpha +\beta \right) i_{k}+\alpha \left(
k-1\right) \right) .
\end{eqnarray*}
\end{theorem}

\begin{proof}
The left hand side $\lah{n+k}{n}^{\alpha ,\beta }$ counts
total weight of distributing a set $1,2,\ldots,n+k$ into $n$ ordered non empty lists.

In the right hand side, we constitute $n$ lists from the elements $1,\ldots
,n$ (one by list having each one a weight $1$). Now, we discuss the weight
of the remaining elements $n+1,\ldots ,n+k$.

To insert the element $n+1$ to a list $i_{1}$ $(1\leq i_{1}\leq n)$, we must
consider all the possible situations of the element already in the list $%
i_{1}$ and we distingue two situations:
\begin{description}
  \item[(A1)] first, the initial element holds in the list $i_{1}$ and we put the
element $n+1$ before the initial one with a weight $\beta $ or after it with
a weight $\alpha $.
  \item[(A2)] second, move the initial element to one of the first $i_{1}-1$ lists with
a weight $\left( \alpha +\beta \right) (i_{1}-1)$ and put the element $n+1$
in the list $i_{1}$ with a weight $1$. Note that, we move the elements only
from right to left to avoid the double counting of situations.
\end{description}

Thus from (A1) and (A2) the weight of the element $n+1$ is $\left( \alpha +\beta \right)
i_{1}$. We sum over all the possible insertions of the element $n+1$, we get
the total weight of the $1,\ldots ,n+1$ elements as $\sum_{1\leq i_{1}\leq
n}\left( \alpha +\beta \right) i_{1}$.

Now, to insert the element $n+2$, we consider the elements of the lists $1,\ldots
,i_{1}$ as fixed ones due to the insertion of the previous element $n+1$
where we consider all the situations. We have two possibilities :
\begin{description}
  \item[(B1)] If we add the element $n+2$ to one of the lists $1,\ldots ,i_{1}$ with
weight $\left( \left( \alpha +\beta \right) i_{1}+\alpha \right) $.

  \item[(B2)] Else, it belongs to a list $i_{2}$ $\left( i_{1}+1\leq i_{2}\leq
n\right) $, with weight $\left( \left( \alpha +\beta \right) i_{2}+\alpha
\right) $ (indeed, the weight of element $n+2$ is $\left( \beta +\alpha \right) $ if it's inserted before or after the initial element of the list $i_{2}$ or $\left(\alpha +\beta \right) (i_{2}-1)+\alpha $ if it's inserted in the list $i_{2}$ and the initial
elements of the list $i_{2}$ is moved to the previous lists).
\end{description}

Than from (B1) and (B2) the weight of the element $n+2$ is
  \begin{equation*}
  \left( \left( \alpha +\beta \right) i_{1}+\alpha \right) +\sum_{i_{2}=i_{1}+1}^{n}\left( \alpha +\beta \right)i_{2}+\alpha= \sum_{i_{2}=i_{1}}^{n}\left( \alpha +\beta \right)i_{2}+\alpha.
  \end{equation*}

Altogether, the weight of the elements $n+1$ and $n+2$ is
\begin{equation*}
\sum_{i_{1}=1}^{n}\left(  \alpha +\beta \right) i_{1}\sum_{i_{2}=i_{1}}^{n}\left( \left( \alpha +\beta \right) i_{2}+\alpha
\right) =\sum_{1\leq i_{1}\leq i_{2}\leq n}\left( \left( \alpha
+\beta \right) i_{1}\right) \left( \left( \alpha +\beta \right) i_{2}+\alpha
\right).
\end{equation*}
We carry on by the same process for the remaining $k-2$ elements. So, for
the last element $n+k$ we consider the elements of the lists $1,\ldots
,i_{k-1}$ as fixed ones, then the weight of the element $n+k$ is  $\left( \alpha +\beta \right)
(i_{k-1}-1)+\alpha (k-1)$ if it's inserted in these lists. Or $\left( \alpha +\beta \right) i_{k}+\alpha (k-1)$
if it's inserted in a list $i_{k}$ $\left( i_{k-1}+1\leq i_{k}\leq n\right) $. This gives the total weight of distributing a set $1,2,\ldots,n+k$ into $n$ ordered non empty lists.
\begin{equation*}
\sum_{1\leq i_{1}\leq i_{2}\leq \cdots \leq i_{k}\leq n}\left( \left( \alpha
+\beta \right) i_{1}\right) \left( \left( \alpha +\beta \right) i_{2}+\alpha
\right) \cdots \left( \left( \alpha +\beta \right) i_{k}+\alpha \left(
k-1\right) \right).
\end{equation*}
\end{proof}

Note that for $\left( \alpha ,\beta \right) =\left( 1,0\right) ,$ $\left(
0,1\right) $ and $\left( 1,1\right) $ we get the identities \cite[eq. 22]%
{MR743795}, \cite[eq. 23]{MR743795} and \cite[eq. 5]{BelBou14} respectively.

\section{Convolution identity}
In this section we proved some combinatorial convolution. The first one is a multinomial convolutional type identity.
\begin{theorem}
The Generalized Stirling numbers satisfy%
\begin{equation*}
\binom{k}{k_{1},\ldots ,k_{p}}\lah{n}{k}^{\alpha ,\beta
}=\sum\limits_{l_{1}+\cdots +l_{p}=n}\binom{n}{l_{1},\ldots ,l_{p}}%
\lah{l_{1}}{k_{1}}^{\alpha ,\beta }\cdots \lah{l_{p}}{k_{p}}^{\alpha ,\beta }.
\end{equation*}
\end{theorem}

\begin{proof}
We consider the weight of the partitions of the set $\{1,\ldots ,n\}$ into $%
k $ lists which is $\lah{n}{k}^{\alpha ,\beta }$. We
color the elements of the lists with $p$ colors such that the elements of
each $k_{i}$ $\left( 1\leq i\leq p\right) $\ lists have the same color,
there are $\binom{k}{k_{1},\ldots ,k_{p}}$ possibilities to do. This is
equivalent to choose the elements of same color then count the weight of
there distribution into lists. So we choose each $l_{i}$ elements that have
the same color and we have $\binom{n}{l_{1},\ldots ,l_{p}}$ possibilities,
then consider the weight of all the distributions of the $l_{i}$ elements
into $k_{i}$ lists and we have $\lah{l_{i}}{k_{i}}^{\alpha ,\beta }$. Summing over all possible values of $%
l_{i} $ gives the result.
\end{proof}

\begin{theorem}
The Generalized Stirling numbers satisfy%
\begin{equation*}
\lah{k+m}{k}^{\alpha ,\beta }=\sum_{j=0}^{s}
\left(\sum_{k-j\leq i_{1}\leq \cdots \leq i_{s-j}\leq k}\prod_{l=0}^{s-j-1} \left( \alpha
+\beta \right) i_{l+1}+\alpha \left( m-\left( s-j-l\right) \right) \right)
\lah{k+m-s}{k-j}^{\alpha ,\beta }.
\end{equation*}
\end{theorem}

\begin{proof}
Let us consider the $s$ $(0\leq s \leq k)$ last elements of the set $\{1,2,\ldots ,k,\ldots ,k+m\}$, we
constitute the last $j$ lists ($0\leq j\leq s$) from these elements. Note that,
the weight of distributed the first $k+m-s$ elements into $k-j$ lists is $%
\lah{k+m-s}{k-j}^{\alpha ,\beta }$. Now, to constitute
the $j$ remaining lists, we pick $j$ elements from the $s$ ones all with weight $%
1 $ and discuss the insertion of the subset of the remaining $s-j$ elements to the $k$ lists. To
insert the first element there are two cases:

1) if we add it in the one of the first $k-j$ lists, then it has the weight $\left( \alpha +\beta \right) \left( k-j\right) +\alpha
\left( m-\left( s-j\right) \right) $ (In fact, $\beta \left( k-j\right) $ as head
list or $\alpha \left( k-j\right) +\alpha m-\left( s-j\right) $ after each elements).

2) Else, we add it to a list $i_{1}$ $(k-j+1\leq i_{1}\leq k)$, we have to
discuss two other cases:

a) the initial element of the list $i_{1}$ stay in the list $i_{1}$ so the
weight to insert the element  $\left( \alpha +\beta \right)$.

b) the initial element moves to one of the previous $i_{1}-1$ lists with
weight $\left( \alpha +\beta \right) (i_{1}-1)+\alpha \left( m-\left(
s-j\right) \right) $.

Thus, we sum over all
the possible insertions in the list $i_{1}$, we get the weight
 $\sum_{k-j+1\leq i_{1}\leq
k}\left( \alpha +\beta \right) i_{1}+\alpha \left( m-\left( s-j\right)
\right) $.

From 1) and 2) we get the weight of the first element
 \begin{equation*}
 \sum_{k-j\leq i_{1}\leq k}\left( \alpha +\beta \right) i_{1}+\alpha
\left( m-\left( s-j\right) \right) .
 \end{equation*}
To insert the second element of the subset, we consider the elements of the lists $%
1,\ldots ,i_{1}$ as fixed ones due to the insertion of the previous element
where we consider all the situations. We have two situations :

a') if we add it to the lists $1,\ldots ,i_{1}$ with weight $\left(
\alpha +\beta \right) i_{1}+\alpha \left( m-\left( s-j\right) +1\right) $.

b') else, it belongs to a list $i_{2}$ $\left( i_{1}+1\leq i_{2}\leq
k\right) $, with weight $\left( \alpha +\beta \right) i_{2}+\alpha \left(
m-\left( s-j\right) +1\right) $, that gives $\sum_{i_{1}+1\leq i_{2}\leq
n}\left( \alpha +\beta \right) i_{2}+\alpha \left( m-\left( s-j\right)
+1\right) $.

Thus, from a') and b') we get
\begin{equation*}
\sum_{k-j\leq i_{1}\leq i_{2}\leq n}\left(
\left( \alpha +\beta \right) i_{1}+\alpha \left( m-\left( s-j\right) \right)
\right) \left( \left( \alpha +\beta \right) i_{2}+\alpha \left( m-\left(
s-j\right) +1\right) \right),
\end{equation*}
by the same way we carry on for the remaining $s-j-2$ elements of the subset, which gives the total weight
\begin{equation*}
\sum_{k-j\leq i_{1}\leq i_{2}\leq \cdots \leq i_{s-j}\leq
k}\left( \left( \alpha +\beta \right) i_{1}+\alpha \left( m-\left(
s-j\right) \right) \right) \cdots \left( \left( \alpha +\beta \right)
i_{s-j}+\alpha \left( m-1\right) \right),
\end{equation*}
then by summing over all possible values of $j$ we get the result.
\end{proof}

\bibliographystyle{abbrv}
\bibliography{BD_Art}

\end{document}